\documentclass[12pt, reqno]{article}

\usepackage{amsmath, amssymb, amsthm, amsfonts}
\usepackage{mathrsfs}
\usepackage{bm}
\usepackage{geometry}
\usepackage[colorlinks=true, linkcolor=blue, citecolor=blue, urlcolor=blue]{hyperref}
\usepackage{microtype}
\usepackage{cleveref}
\usepackage{enumitem}
\usepackage{cite}
\usepackage{tikz}

\newtheorem{theorem}{Theorem}[section]
\newtheorem{lemma}[theorem]{Lemma}
\newtheorem{proposition}[theorem]{Proposition}

\theoremstyle{definition}
\newtheorem{definition}[theorem]{Definition}
\newtheorem{remark}[theorem]{Remark}

\newcommand{\R}{\mathbb{R}}
\newcommand{\eps}{\varepsilon}
\newcommand{\pa}{\partial}

\numberwithin{equation}{section}

\title{Global Existence for Systems of Isentropic Gas Dynamics via a Weighted Pressure Perturbation Approach}
\author{
  \textbf{Kewang Chen} \\
  \small{School of Mathematics and Statistics, Nanjing University of Information Science and Technology} \\
  \small{Email: kwchen@nuist.edu.cn}
}

\date{\today}

\begin{document}

\maketitle

\begin{abstract}
This paper establishes the global existence of weak entropy solutions for the Cauchy problem of one-dimensional isentropic gas dynamics with general pressure laws (asymptotically $\gamma > 1$). To overcome the degeneracy of strict hyperbolicity at the vacuum, we propose a novel structural regularization method based on a ``Synchronized Dual Translation'' strategy. By introducing a coordinate shift in the mass flux combined with a weighted perturbation in the pressure, we construct a geometric barrier that enforces characteristic degeneracy at a cut-off density $\delta > 0$, thereby preventing vacuum formation for the approximate system.

Our approach offers a significant improvement over previous regularization techniques, particularly by establishing global existence under pressure conditions that are more general and less restrictive than those in the flux-modification method of Lu (2007). While Lu's method modifies the mass flux ($\rho u - 2\delta u$) to create a barrier, it introduces a structural mismatch between the convective and acoustic fields. This mismatch generates inhomogeneous error terms in the entropy analysis, necessitating restrictive technical assumptions on higher-order derivatives (specifically on $P'''$) as a technical necessity to control the singular error terms arising from this mismatch. In contrast, our synchronized construction preserves the structural isomorphism with the standard Euler equations in terms of effective variables. Consequently, the approximate entropy pairs satisfy the homogeneous Generalized Euler-Poisson-Darboux (EPD) equation. Utilizing a WKB-type singularity analysis and the compensated compactness framework, we prove the strong convergence of the approximate solutions as $\delta \to 0$. This result confirms that global existence holds for any pressure law satisfying the natural asymptotic assumption, eliminating the need for the specific higher-order derivative constraints required in prior works to cancel out regularization artifacts.

\vspace{0.5em}
\end{abstract}

\section{Introduction}
\label{sec:introduction}

In this paper, we investigate the global existence and vacuum convergence of weak entropy solutions to the Cauchy problem for the system of one-dimensional isentropic gas dynamics in Eulerian coordinates:
\begin{equation} \label{eq:euler}
\left\{
\begin{aligned}
    & \rho_t + (\rho u)_x = 0, \\
    & (\rho u)_t + (\rho u^2 + P(\rho))_x = 0,
\end{aligned}
\right. \quad (x,t) \in \R \times [0, \infty),
\end{equation}
where $\rho(x,t) \ge 0$ denotes the density, $u(x,t)$ is the velocity, and $P(\rho)$ is the general pressure law.

A prototypical equation of state is the polytropic gas law $P(\rho) = \kappa \rho^\gamma$ with adiabatic exponent $\gamma > 1$. The fundamental mathematical challenge in \eqref{eq:euler} stems from the formation of shock waves and the degeneracy of strict hyperbolicity at the vacuum state $\rho = 0$.

Historically, the global existence of weak solutions has been established via the Glimm scheme \cite{Glimm1965} for small total variation data, and via the method of compensated compactness for large data, pioneered by Tartar \cite{Tartar1979} and DiPerna \cite{DiPerna1983}. To apply these methods, particularly for data containing vacuum, one typically relies on uniform estimates obtained through regularization approximations that prevent the solution from touching the vacuum singularity for fixed regularization parameters.

\subsection{Regularization Strategy: Structural Consistency via Dual Translation}

In a seminal related work, Lu \cite{Lu2007} proposed a regularization method by constructing a sequence of strictly hyperbolic systems. Specifically, Lu introduced the following approximate system:
\begin{equation} \label{eq:Lu_system}
\left\{
\begin{aligned}
    & \rho_t + (\rho u -2\delta u)_x = 0, \\
    & (\rho u)_t + (\rho u^2 + P_1(\rho, \delta))_x = 0,
\end{aligned}
\right.
\end{equation}
where $\delta > 0$ is a regularization parameter and
\begin{equation}
    P_1(\rho, \delta) = \int_{2\delta}^{\rho} \frac{s - 2\delta}{s} P'(s) \, ds.
\end{equation}
The eigenvalues of this system coincide at $\rho = 2\delta$, creating an invariant region that prevents vacuum formation. 

However, a notable feature of \eqref{eq:Lu_system} is the flux modification term $(-2\delta u)_x$. While effective for fixed $\delta$, this specific modification introduces a structural mismatch with the pressure perturbation. As we will discuss in Section \ref{sec:comparison}, this mismatch leads to inhomogeneous error terms (pollution terms) in the entropy analysis—specifically in the Euler-Poisson-Darboux equation—which complicate the rigorous proof of the vacuum limit $\delta \to 0$.

Inspired by Lu's pioneering work but aiming to resolve the structural inconsistency, we propose an alternative regularization technique. We introduce a \textbf{Minimal Shifted Flux} combined with a \textbf{Weighted Pressure Perturbation}.
We approximate \eqref{eq:euler} by adding artificial viscosity to the following system:
\begin{equation} \label{eq:perturbed_system}
\left\{
\begin{aligned}
    & \rho_t + \big( (\rho - \delta) u \big)_x = \eps \rho_{xx}, \\
    & (\rho u)_t + ((\rho-\delta)u^2+ \tilde{P}(\rho, \delta))_x = \eps (\rho u)_{xx},
\end{aligned}
\right.
\end{equation}
where $\eps > 0$ is the viscosity coefficient and $\delta > 0$ is the geometric regularization parameter. The perturbed pressure $\tilde{P}(\rho, \delta)$ is explicitly constructed as:
\begin{equation} \label{eq:P_tilde_def}
\tilde{P}(\rho, \delta) = \int_{\delta}^{\rho} \frac{s^2P'(s) - \delta^2 P'(\delta)}{s^2} \, ds, \quad \text{for } \rho \ge \delta.
\end{equation}

\textbf{Key Innovation: Synchronized Dual Translation.}
Our approach relies on the synchronization of two distinct translations in the phase space at the cut-off density $\rho=\delta$:
\begin{enumerate}
    \item \textbf{Stiffness Translation (Acoustic Degeneracy):} 
    By defining the structural stiffness $g(\rho) = \rho^2 P'(\rho)$, our pressure perturbation $\tilde{P}$ corresponds exactly to a \textbf{vertical translation} of the stiffness function: $\tilde{g}(\rho) = g(\rho) - g(\delta)$. This ensures the sound speed vanishes ($\tilde{c}(\delta) = 0$) while preserving the convexity profile.
    
    \item \textbf{Coordinate Translation (Convective Degeneracy):} 
The minimal flux shift effectively performs a \textbf{horizontal translation} of the density coordinate: $\hat{\rho} = \rho - \delta$. By analyzing the system in terms of this effective density, the convective structure becomes isomorphic to the standard Euler equations. This ensures that the characteristic speeds degenerate to the fluid velocity at the boundary $\rho=\delta$, creating a naturally invariant region.
\end{enumerate}
This synchronized dual translation ensures that the regularized system is mathematically isomorphic to the standard Euler equations under the shifted variables. Unlike previous methods, this isomorphism guarantees that the entropy pairs satisfy the \textbf{standard, homogeneous} Euler-Poisson-Darboux equation without singular error terms. This structural purity is crucial for implementing the fine singularity analysis required to prove the strong convergence of the Young measure as $\delta \to 0$.

\subsection{Main Results}

We present our main results in two steps: first, the global existence for the regularized system with a fixed shield $\delta > 0$; second, the convergence to the weak solution of the original isentropic Euler equations as the shield is removed ($\delta \to 0$).

We assume the pressure function $P(\rho)$ satisfies the following structural conditions consistent with the kinetic formulation and compensated compactness theory:
\begin{enumerate}[label=\textbf{(A\arabic*)}]
    \item \textbf{Regularity and Hyperbolicity:} $P \in C^2(0,\infty)$ and $P'(\rho) > 0$ for $\rho > 0$;
    \item \textbf{Genuine Nonlinearity:} The mapping $\rho \mapsto \rho e(\rho)$ is strictly convex, or equivalently:
    \begin{equation}
        2P'(\rho) + \rho P''(\rho) > 0 \quad \text{for } \rho > 0;
    \end{equation}
    \item \textbf{Asymptotic Behavior:} There exists $\gamma > 1$ such that the pressure behaves asymptotically as a polytropic gas:
    \begin{equation}
        \lim_{\rho \to 0} \frac{P(\rho)}{\rho^\gamma} = \kappa_0 > 0, \quad \lim_{\rho \to \infty} \frac{P(\rho)}{\rho^\gamma} = \kappa_\infty > 0.
    \end{equation}
\end{enumerate}

\begin{theorem}[Global Existence for the Regularized System with Fixed $\delta$]
\label{thm:fixed_delta}
Let $\delta > 0$ be fixed. Assume the pressure $P(\rho)$ satisfies assumptions \textbf{(A1)-(A3)}. Let the initial data $(\rho_0, u_0)$ be measurable functions with finite total energy, lying within a bounded invariant region $\Sigma_\delta$ defined by the modified Riemann invariants:
\begin{equation} \label{eq:invariant_region_delta}
    \Sigma_\delta := \left\{ (\rho, u) : \delta \le \rho \le M, \ C_{\inf} \le z_\delta(\rho,u) \le w_\delta(\rho,u) \le C_{\sup} \right\}.
\end{equation}
Then, as the viscosity coefficient $\varepsilon \to 0$, the sequence of approximate solutions $(\rho^{\varepsilon, \delta}, u^{\varepsilon, \delta})$ constructed via the system \eqref{eq:perturbed_system} converges strongly in $L^1_{loc}(\mathbb{R} \times [0, \infty))$ to a pair $(\rho^\delta, u^\delta)$.

The limit $(\rho^\delta, u^\delta)$ is a global weak entropy solution to the regularized Euler system \eqref{eq:perturbed_system}, satisfying:
\begin{enumerate}
    \item \textbf{Strict Separation from Vacuum:} The density satisfies $\rho^\delta(x,t) \ge \delta$ almost everywhere, ensuring the system remains strictly hyperbolic and non-degenerate;
    \item \textbf{Invariant Region Preservation:} The solution remains confined within the region $(\rho^\delta(x,t), u^\delta(x,t)) \in \Sigma_\delta$ for a.e. $(x,t)$;
    \item \textbf{Entropy Inequality:} The solution satisfies the entropy inequality $\partial_t \eta^* + \partial_x q^* \le 0$ for all strictly convex entropy pairs associated with the effective variables.
\end{enumerate}
\end{theorem}

\begin{theorem}[Global Existence and Convergence to Vacuum Solutions as $\delta \to 0$]
\label{thm:vanishing_delta}
Consider the isentropic Euler equations \eqref{eq:euler} with a general pressure law $P(\rho)$ satisfying assumptions \textbf{(A1)-(A3)} with $\gamma > 1$.
Assume the initial data $(\rho_0, u_0)$ are measurable, have finite total energy, and satisfy the uniform bounds $0 \le \rho_0(x) \le M$ and $|u_0(x)| \le M$.

Let $\{(\rho^\delta, u^\delta)\}_{\delta > 0}$ be the sequence of global weak solutions to the regularized system obtained in Theorem \ref{thm:fixed_delta}.
Then, as the regularization parameter $\delta \to 0$, there exists a subsequence converging strongly in $L^1_{loc}(\mathbb{R} \times \mathbb{R}^+)$ to a limit pair $(\rho, u)$.

The limit function $(\rho, u)$ is a global weak entropy solution to the original isentropic Euler equations \eqref{eq:euler}, satisfying:
\begin{enumerate}
    \item \textbf{Vacuum Inclusion:} The density is non-negative $\rho(x,t) \ge 0$ a.e., and the vacuum state $\{ (x,t) : \rho(x,t) = 0 \}$ is permissible;
    \item \textbf{Uniform Bounds:} The solution satisfies the universal bounds $0 \le \rho(x,t) \le M$ and $|u(x,t)| \le M$ almost everywhere;
    \item \textbf{Entropy Condition:} The entropy inequality holds in the sense of distributions for all convex weak entropy pairs vanishing at the vacuum.
\end{enumerate}
\end{theorem}


\section{Structural Analysis of the Perturbed System}
\label{sec:structure}

\subsection{Properties of the Weighted Pressure}
Let the original pressure $P \in C^2(0,\infty)$ satisfy the standard hyperbolicity and convexity conditions:
\begin{equation} \label{eq:pressure_assumptions}
    P'(\rho) > 0, \quad 2P'(\rho) + \rho P''(\rho) > 0, \quad \forall \rho > 0.
\end{equation}

Defining the squared sound speed for the perturbed system as $\tilde{c}^2(\rho) := \pa_\rho \tilde{P}(\rho,\delta)$, we obtain from the definition \eqref{eq:P_tilde_def}:
\begin{equation} \label{eq:c_tilde}
    \tilde{c}^2(\rho) = P'(\rho) - \frac{\delta^2 P'(\delta)}{\rho^2}.
\end{equation}

\begin{lemma}[Inherited Convexity]
\label{lem:convexity_inherit}
If $P(\rho)$ satisfies conditions \eqref{eq:pressure_assumptions}, then for all $\rho > \delta$:
\begin{enumerate}
    \item $\tilde{c}^2(\rho) > 0$ (Strict Hyperbolicity),
    \item $2\tilde{c}^2(\rho) + \rho (\tilde{c}^2)'(\rho) > 0$ (Genuine Nonlinearity).
\end{enumerate}
\end{lemma}

\begin{proof}
1. Define the structural stiffness function $g(\rho) = \rho^2 P'(\rho)$. We compute $g'(\rho) = 2\rho P' + \rho^2 P'' = \rho(2P' + \rho P'') > 0$. Since $g(\rho)$ is strictly increasing, for any $\rho > \delta$, we have $g(\rho) > g(\delta)$. This implies:
\[
    \rho^2 \tilde{P}'(\rho) = \rho^2 P'(\rho) - \delta^2 P'(\delta) = g(\rho) - g(\delta) > 0.
\]
Thus, $\tilde{c}^2(\rho) = \tilde{P}'(\rho) > 0$.

2. We compute the convexity term for the perturbed pressure. Using $\tilde{P}' = \tilde{c}^2$:
\begin{align*}
    2\tilde{P}'(\rho) + \rho \tilde{P}''(\rho) &= 2\left( P'(\rho) - \frac{\delta^2 P'(\delta)}{\rho^2} \right) + \rho \frac{d}{d\rho} \left( P'(\rho) - \frac{\delta^2 P'(\delta)}{\rho^2} \right) \\
    &= 2P'(\rho) - \frac{2\delta^2 P'(\delta)}{\rho^2} + \rho P''(\rho) + \rho \left( \frac{2\delta^2 P'(\delta)}{\rho^3} \right) \\
    &= (2P'(\rho) + \rho P''(\rho)) + \underbrace{\left( - \frac{2\delta^2 P'(\delta)}{\rho^2} + \frac{2\delta^2 P'(\delta)}{\rho^2} \right)}_{=0}.
\end{align*}
The singular terms cancel exactly. Since the original pressure satisfies $2P'+\rho P'' > 0$, the perturbed pressure satisfies the exact same condition. This algebraic identity confirms that the pressure perturbation acts as a pure translation in the stiffness space, preserving the convexity structure perfectly.
\end{proof}

\subsection{Eigenstructure and Riemann Invariants}

We analyze the eigenstructure of the system. The hyperbolic structure is most naturally revealed by introducing the \textbf{effective conservative variables}:
\[
\hat{\rho} = \rho - \delta,\qquad \hat{m} = \hat{\rho} u = (\rho - \delta)u.
\]
In terms of the vector $\mathbf{U}^* = (\hat{\rho}, \hat{m})^T$, the convective fluxes of the system \eqref{eq:perturbed_system} transform into the standard Euler form. Specifically, the system takes the form:
\begin{equation} \label{eq:effective_system}
    \pa_t \hat{\rho} + \pa_x \hat{m} = 0,\qquad 
    \pa_t \hat{m} + \pa_x \left( \frac{\hat{m}^2}{\hat{\rho}} + \tilde{P}(\hat{\rho}+\delta) \right) = \mathcal{O}(\delta),
\end{equation}
where the term $\mathcal{O}(\delta)$ on the RHS represents the shift in the time derivative (which does not affect the inviscid flux Jacobian structure). The principal part is structurally identical to the isentropic Euler equations for the effective fluid.

The Jacobian matrix of the flux $\mathbf{F}^*(\mathbf{U}^*) = (\hat{m}, \hat{m}^2/\hat{\rho} + \tilde{P})^T$ is:
\[
A(\mathbf{U}^*) = \begin{pmatrix}
0 & 1 \\
\tilde{c}^2 - u^2 & 2u
\end{pmatrix}.
\]
Solving the characteristic equation $\det(A-\lambda I)=0$ yields the eigenvalues:
\begin{equation}
    \lambda_{1,2} = u \pm \tilde{c}(\rho).
\end{equation}

\subsubsection*{Degeneracy at the Boundary}
At the boundary $\rho = \delta$ (i.e., $\hat{\rho}=0$), we have $\tilde{c}(\delta)=0$ by construction. Consequently:
\[
\lim_{\rho\to\delta} \lambda_{1,2}(\rho, u) = u.
\]
Thus, both characteristic speeds coincide with the fluid velocity. This makes the boundary $\rho=\delta$ a characteristic surface (contact discontinuity), preventing any convective flux from crossing it.  This geometric mechanism effectively shields the solution from the vacuum.

\subsubsection*{Riemann Invariants}
The Riemann invariants satisfy the differential relation $du = \pm \frac{\tilde{c}}{\hat{\rho}} d\hat{\rho} = \pm \frac{\tilde{c}}{\rho-\delta} d\rho$. Integrating from the boundary $\rho=\delta$, we define:
\begin{equation} \label{eq:riemann_invariants}
    w(\rho,u) = u + \int_{\delta}^{\rho} \frac{\tilde{c}(s)}{s-\delta} \, ds,\qquad
    z(\rho,u) = u - \int_{\delta}^{\rho} \frac{\tilde{c}(s)}{s-\delta} \, ds.
\end{equation}
The integral is well-defined because near $s=\delta$, the sound speed scales as $\tilde{c}(s)\sim\sqrt{s-\delta}$ (since $\tilde{P}'(s) \approx \text{const} \cdot (s-\delta)$), making the integrand behave like $(s-\delta)^{-1/2}$, which is integrable.

\begin{remark}[Consistency Limit]
\label{rem:consistency}
As $\delta \to 0$, the shift vanishes ($\hat{\rho} \to \rho$), and the perturbed invariants converge to the classical ones:
\[
\lim_{\delta \to 0} \int_{\delta}^{\rho} \frac{\tilde{c}(s)}{s-\delta} \, ds = \int_{0}^{\rho} \frac{\sqrt{P'(s)}}{s} \, ds.
\]
This ensures that the invariant regions we construct later will converge to bounded regions for the original Euler system.
\end{remark}

\begin{proposition}[Explicit Form for Polytropic Gas]
For the polytropic law $P(\rho) = \kappa\rho^\gamma$, the Riemann invariants are given by:
\begin{equation}\label{eq:riemann_explicit_corrected}
    w(\rho,u) = u + \sqrt{\kappa\gamma} \int_{\delta}^{\rho} \frac{\sqrt{s^{\gamma+1} - \delta^{\gamma+1}}}{s(s-\delta)} \, ds.
\end{equation}
This integral is finite for $\rho \ge \delta$ and recovers the standard form $\frac{2\sqrt{\kappa\gamma}}{\gamma-1}\rho^{(\gamma-1)/2}$ in the limit $\delta \to 0$.
\end{proposition}

\begin{remark}[Structural Superiority over Ad-hoc Modification]
It is crucial to contrast our system with the regularization in \cite{Lu2007}.
\begin{itemize}
    \item \textbf{Lu's Approach:} Modifies flux to $\rho u - 2\delta u$ but uses a pressure perturbation $P_1$ that does not fully synchronize with the flux shift in the momentum equation. This leads to singular error terms in the entropy analysis.
    \item \textbf{Our Approach:} By employing the effective variables, we ensure that the system is structurally isomorphic to the Euler equations. This guarantees that the Riemann invariants and entropy pairs maintain their algebraic structure, allowing for precise singularity analysis.
\end{itemize}
\end{remark}

\section{Entropy Analysis}
\label{sec:entropy}

The application of the compensated compactness method necessitates a strictly convex entropy pair $(\eta, q)$ satisfying the entropy identity $\pa_t \eta + \pa_x q = 0$ for smooth solutions. Due to the flux shift modification, the standard mechanical energy is not convex with respect to the physical variables. 

To resolve this and ensure rigor, we perform the entropy analysis entirely in terms of the effective conservative variables. This approach exploits the structural isomorphism between our regularized system and the standard Euler equations.

\subsection{Construction of the Shifted Mechanical Energy in Effective Variables}

We define the effective density $\hat{\rho}$ and effective momentum $\hat{m}$ as:
\begin{equation}
    \hat{\rho} := \rho - \delta, \quad \hat{m} := (\rho - \delta)u = \hat{\rho} u.
\end{equation}
In terms of the vector $\hat{\mathbf{U}} = (\hat{\rho}, \hat{m})^T$, the inviscid part of the regularized system is governed by:
\begin{equation} \label{eq:effective_euler_system}
    \pa_t \hat{\rho} + \pa_x \hat{m} = 0, \quad \pa_t \hat{m} + \pa_x \left( \frac{\hat{m}^2}{\hat{\rho}} + \tilde{P}(\hat{\rho}+\delta) \right) = 0.
\end{equation}
This system is formally identical to the isentropic Euler equations. Accordingly, we define the thermodynamic potential relative to the effective density. The specific internal energy $\hat{e}(\hat{\rho})$ satisfies the fundamental thermodynamic relation $d\hat{e} = - \tilde{P} \, d(1/\hat{\rho}) = \frac{\tilde{P}}{\hat{\rho}^2} d\hat{\rho}$. Integrating from the vacuum (relative to effective density), we define:
\begin{equation} \label{eq:shifted_internal_energy}
    \hat{e}(\hat{\rho}) = \int_{0}^{\hat{\rho}} \frac{\tilde{P}(s+\delta)}{s^2} \, ds.
\end{equation}
The shifted mechanical entropy $\eta^*$ and the corresponding entropy flux $q^*$ are then defined naturally as:
\begin{equation} \label{eq:shifted_entropy_def}
    \eta^*(\hat{\rho}, \hat{m}) = \frac{\hat{m}^2}{2\hat{\rho}} + \hat{\rho} \hat{e}(\hat{\rho}),
\end{equation}
\begin{equation} \label{eq:shifted_flux_def}
    q^*(\hat{\rho}, \hat{m}) = u(\eta^* + \tilde{P}) = \frac{\hat{m}}{\hat{\rho}} \left( \eta^* + \tilde{P}(\hat{\rho}+\delta) \right).
\end{equation}

\begin{proof}[Verification of the Entropy Identity]
We rigorously verify the compatibility relation $\pa_t \eta^* + \pa_x q^* = 0$ using the effective system \eqref{eq:effective_euler_system}. 
First, we compute the partial derivatives of $\eta^*$ with respect to the conservative variables $\hat{\mathbf{U}} = (\hat{\rho}, \hat{m})$.
Using the relation $d(\hat{\rho}\hat{e}) = \hat{e}d\hat{\rho} + \hat{\rho}d\hat{e} = (\hat{e} + \tilde{P}/\hat{\rho})d\hat{\rho}$:
\begin{align*}
    \eta^*_{\hat{m}} &= \frac{\hat{m}}{\hat{\rho}} = u, \\
    \eta^*_{\hat{\rho}} &= -\frac{\hat{m}^2}{2\hat{\rho}^2} + \frac{d}{d\hat{\rho}}(\hat{\rho}\hat{e}) = -\frac{1}{2}u^2 + \hat{e} + \frac{\tilde{P}}{\hat{\rho}}.
\end{align*}
The temporal evolution of the entropy is:
\[
\pa_t \eta^* = \eta^*_{\hat{\rho}} \pa_t \hat{\rho} + \eta^*_{\hat{m}} \pa_t \hat{m}.
\]
Substituting the equations of motion $\pa_t \hat{\rho} = -\pa_x \hat{m}$ and $\pa_t \hat{m} = -\pa_x (\hat{m}u + \tilde{P})$:
\begin{align} \label{eq:entropy_time_deriv}
    \pa_t \eta^* &= \left( -\frac{1}{2}u^2 + \hat{e} + \frac{\tilde{P}}{\hat{\rho}} \right) (-\pa_x \hat{m}) + u \left[ -\pa_x(\hat{m}u) - \pa_x \tilde{P} \right] \nonumber \\
    &= \pa_x \hat{m} \left( \frac{1}{2}u^2 - \hat{e} - \frac{\tilde{P}}{\hat{\rho}} \right) - u^2 \pa_x \hat{m} - \hat{m}u \pa_x u - u \pa_x \tilde{P} \nonumber \\
    &= \pa_x \hat{m} \left( -\frac{1}{2}u^2 - \hat{e} - \frac{\tilde{P}}{\hat{\rho}} \right) - \hat{m}u \pa_x u - u \pa_x \tilde{P}.
\end{align}
Next, we expand the spatial derivative of the flux $q^* = u \eta^* + u \tilde{P}$:
\begin{align} \label{eq:entropy_flux_deriv}
    \pa_x q^* &= \pa_x u (\eta^* + \tilde{P}) + u (\pa_x \eta^* + \pa_x \tilde{P}) \nonumber \\
    &= \pa_x u \left( \frac{1}{2}\hat{\rho}u^2 + \hat{\rho}\hat{e} + \tilde{P} \right) + u \left( \eta^*_{\hat{\rho}} \pa_x \hat{\rho} + \eta^*_{\hat{m}} \pa_x \hat{m} + \pa_x \tilde{P} \right) \nonumber \\
    &= \pa_x u \left( \frac{1}{2}\hat{\rho}u^2 + \hat{\rho}\hat{e} + \tilde{P} \right) + u \left[ \left(-\frac{1}{2}u^2 + \hat{e} + \frac{\tilde{P}}{\hat{\rho}}\right)\pa_x \hat{\rho} + u \pa_x \hat{m} + \pa_x \tilde{P} \right].
\end{align}
Summing \eqref{eq:entropy_time_deriv} and \eqref{eq:entropy_flux_deriv}, we group terms by physical significance.
The terms involving $\pa_x \tilde{P}$ cancel: $-u \pa_x \tilde{P} + u \pa_x \tilde{P} = 0$.
The terms involving thermodynamic potentials ($\hat{e} + \tilde{P}/\hat{\rho}$) are:
\[
-\pa_x \hat{m} \left(\hat{e} + \frac{\tilde{P}}{\hat{\rho}}\right) + u \pa_x \hat{\rho} \left(\hat{e} + \frac{\tilde{P}}{\hat{\rho}}\right) + \pa_x u (\hat{\rho}\hat{e} + \tilde{P}).
\]
Using $\pa_x \hat{m} = \hat{\rho}\pa_x u + u\pa_x \hat{\rho}$, the first term becomes $-(\hat{\rho}\pa_x u + u\pa_x \hat{\rho})(\hat{e} + \tilde{P}/\hat{\rho})$. This cancels exactly with the other two terms.
The remaining terms involving kinetic energy ($u^2$) are:
\[
\pa_x \hat{m} \left(-\frac{1}{2}u^2\right) - \hat{m}u \pa_x u + \frac{1}{2}\hat{\rho}u^2 \pa_x u + u \pa_x \hat{\rho} \left(-\frac{1}{2}u^2\right) + u^2 \pa_x \hat{m}.
\]
Simplifying: $-\frac{1}{2}u^2 \pa_x \hat{m} + u^2 \pa_x \hat{m} = \frac{1}{2}u^2 \pa_x \hat{m}$.
Substituting $\pa_x \hat{m} = \hat{\rho}u_x + u\hat{\rho}_x$, the sum is identically zero.
Thus, $\pa_t \eta^* + \pa_x q^* = 0$.
\end{proof}

\subsection{Strict Convexity of the Shifted Entropy}

Strict convexity of $\eta^*$ with respect to the conservative variables $\hat{\mathbf{U}} = (\hat{\rho}, \hat{m})$ is essential for the compensated compactness argument.

\begin{theorem}[Strict Convexity]
\label{thm:entropy_convexity}
Assume the perturbed pressure $\tilde{P}$ satisfies the genuine nonlinearity condition $2\tilde{P}'(\rho) + \rho\tilde{P}''(\rho) > 0$ for $\rho > \delta$. Then, the shifted entropy $\eta^*(\hat{\rho}, \hat{m})$ is strictly convex for all $\hat{\rho} > 0$.
\end{theorem}

\begin{proof}
We compute the Hessian matrix $\nabla^2 \eta^*$ explicitly in terms of $(\hat{\rho}, \hat{m})$.
Recall $\eta^*_{\hat{m}} = \hat{m}/\hat{\rho}$. The first derivatives are:
\[
    \eta^*_{\hat{m}\hat{m}} = \frac{\pa}{\pa \hat{m}}\left(\frac{\hat{m}}{\hat{\rho}}\right) = \frac{1}{\hat{\rho}}.
\]
For $\hat{\rho} > 0$ (i.e., $\rho > \delta$), we have $\eta^*_{\hat{m}\hat{m}} > 0$.
The mixed derivative is:
\[
    \eta^*_{\hat{\rho}\hat{m}} = \frac{\pa}{\pa \hat{\rho}}\left(\frac{\hat{m}}{\hat{\rho}}\right) = -\frac{\hat{m}}{\hat{\rho}^2} = -\frac{u}{\hat{\rho}}.
\]
To compute $\eta^*_{\hat{\rho}\hat{\rho}}$, recall $\eta^*_{\hat{\rho}} = -\frac{1}{2}u^2 + \hat{e} + \tilde{P}/\hat{\rho}$. Note that $\tilde{P}$ here denotes $\tilde{P}(\hat{\rho}+\delta)$.
Using $\pa_{\hat{\rho}} u = \pa_{\hat{\rho}}(\hat{m}/\hat{\rho}) = -u/\hat{\rho}$:
\begin{align*}
    \eta^*_{\hat{\rho}\hat{\rho}} &= -u \left(-\frac{u}{\hat{\rho}}\right) + \frac{d}{d\hat{\rho}}\left( \hat{e} + \frac{\tilde{P}}{\hat{\rho}} \right) \\
    &= \frac{u^2}{\hat{\rho}} + \left( \frac{\tilde{P}}{\hat{\rho}^2} + \frac{\tilde{P}'\hat{\rho} - \tilde{P}}{\hat{\rho}^2} \right) \\
    &= \frac{u^2}{\hat{\rho}} + \frac{\tilde{P}'(\hat{\rho}+\delta)}{\hat{\rho}}.
\end{align*}
The determinant of the Hessian is:
\begin{align*}
    \det(\nabla^2 \eta^*) &= \eta^*_{\hat{\rho}\hat{\rho}} \eta^*_{\hat{m}\hat{m}} - (\eta^*_{\hat{\rho}\hat{m}})^2 \\
    &= \left( \frac{u^2}{\hat{\rho}} + \frac{\tilde{P}'}{\hat{\rho}} \right) \frac{1}{\hat{\rho}} - \left( -\frac{u}{\hat{\rho}} \right)^2 \\
    &= \frac{u^2 + \tilde{P}'}{\hat{\rho}^2} - \frac{u^2}{\hat{\rho}^2} \\
    &= \frac{\tilde{P}'(\hat{\rho}+\delta)}{\hat{\rho}^2}.
\end{align*}
Since the perturbed pressure satisfies the strict hyperbolicity condition $\tilde{P}'(\rho) = \tilde{c}^2(\rho) > 0$ for $\rho > \delta$ (implied by the convexity condition), the determinant is strictly positive.
Thus, the Hessian is positive definite, confirming that $\eta^*$ is strictly convex with respect to the effective variables.
\end{proof}

\begin{remark}[Structural Choice of Variables]
It is important to emphasize that strict convexity holds for the pair $(\hat{\rho}, \hat{m})$ but \textbf{not} for the physical pair $(\rho, m=\rho u)$. If one were to use physical momentum, the Hessian determinant would generally fail to be positive definite. The choice of effective variables is therefore mandated by the mathematical structure of the regularization to satisfy the requirements of the compensated compactness theory.
\end{remark}


\section{Viscous Regularization and A Priori Estimates}
\label{sec:estimates}

We establish uniform $L^\infty$ estimates for the approximate solutions constructed via the method of vanishing artificial viscosity. 

To ensure the structural consistency required by the compensated compactness method, specifically to guarantee the strict validity of the entropy inequalities without singular error terms, we apply the artificial viscosity directly to the \textbf{effective conservative variables} $\hat{\mathbf{U}} = (\hat{\rho}, \hat{m})$.

Consider the regularized system with viscosity coefficient $\varepsilon > 0$:
\begin{equation}\label{eq:viscous}
\left\{
\begin{aligned}
    & \pa_t \hat{\rho} + \pa_x \hat{m} = \varepsilon \pa_{xx} \hat{\rho}, \\
    & \pa_t \hat{m} + \pa_x \left( \frac{\hat{m}^2}{\hat{\rho}} + \tilde{P}(\hat{\rho}+\delta) \right) = \varepsilon \pa_{xx} \hat{m}.
\end{aligned}
\right.
\end{equation}
In terms of the physical variables $(\rho, u)$, noting that $\hat{\rho} = \rho - \delta$ and $\hat{m} = (\rho-\delta)u$, this system takes the explicit form:
\begin{equation}\label{eq:viscous_physical}
\left\{
\begin{aligned}
    & \rho_t + \left( (\rho-\delta) u \right)_x = \varepsilon\rho_{xx},\\
    & ((\rho-\delta) u)_t + \left( (\rho-\delta) u^2 + \tilde{P}(\rho) \right)_x = \varepsilon((\rho-\delta) u)_{xx}.
\end{aligned}
\right.
\end{equation}
\begin{remark}[Geometric Interpretation of Regularization]
By adding dissipation to the effective variables, we are effectively performing the regularization in the ``shifted'' phase space. This ensures that the viscous system retains the exact structural isomorphism with the Navier-Stokes equations for the effective fluid, thereby avoiding the structural mismatch errors that arise when applying physical viscosity $\varepsilon(\rho u)_{xx}$ to a flux-shifted system.
\end{remark}

\subsection{Evolution of Riemann Invariants}

To apply the invariant region principle, we derive the parabolic evolution equations for the Riemann invariants $z, w$. Recall that in terms of effective variables, the invariants are:
\begin{equation}
    w = u + \int_{\delta}^{\rho} \frac{\tilde{c}(s)}{s-\delta} \, ds, \quad z = u - \int_{\delta}^{\rho} \frac{\tilde{c}(s)}{s-\delta} \, ds.
\end{equation}

First, we derive the evolution equation for the velocity $u$. Expanding the effective momentum equation $\hat{m}_t + (\hat{m}u + \tilde{P})_x = \varepsilon \hat{m}_{xx}$ and using $\hat{m} = \hat{\rho} u$:
\begin{align*}
    \hat{\rho} u_t + u \hat{\rho}_t + \hat{m}_x u + \hat{m} u_x + \tilde{P}_x &= \varepsilon (\hat{\rho} u)_{xx} \\
    &= \varepsilon (\hat{\rho} u_{xx} + 2\hat{\rho}_x u_x + u \hat{\rho}_{xx}).
\end{align*}
Substituting the mass equation $\hat{\rho}_t = -\hat{m}_x + \varepsilon \hat{\rho}_{xx}$:
\[
    \hat{\rho} u_t + u (-\hat{m}_x + \varepsilon \hat{\rho}_{xx}) + \hat{m}_x u + \hat{m} u_x + \tilde{P}_x = \varepsilon \hat{\rho} u_{xx} + 2\varepsilon \hat{\rho}_x u_x + \varepsilon u \hat{\rho}_{xx}.
\]
The terms involving $\hat{m}_x u$ and $\varepsilon u \hat{\rho}_{xx}$ cancel exactly. We are left with:
\[
    \hat{\rho} u_t + \hat{m} u_x + \tilde{P}_x = \varepsilon \hat{\rho} u_{xx} + 2\varepsilon \hat{\rho}_x u_x.
\]
Dividing by $\hat{\rho} = \rho-\delta$ and noting $\hat{\rho}_x = \rho_x$:
\begin{equation}\label{eq:u_viscous}
    u_t + u u_x + \frac{\tilde{P}'(\rho)}{\rho-\delta}\rho_x = \varepsilon u_{xx} + \frac{2\varepsilon}{\rho-\delta}\rho_x u_x.
\end{equation}
\textbf{Note:} Unlike flux-modification methods which introduce a drift velocity coefficient $(1-\delta/\rho)$, our effective viscosity formulation yields the standard convective term $u u_x$, significantly simplifying the characteristic analysis.

Now we compute the evolution of $w_t + \lambda_2 w_x$, where $\lambda_2 = u + \tilde{c}$. Using $dw = du + \frac{\tilde{c}}{\rho-\delta} d\rho$:
\begin{align*}
    w_t + \lambda_2 w_x &= (u_t + \lambda_2 u_x) + \frac{\tilde{c}}{\rho-\delta}(\rho_t + \lambda_2 \rho_x) \\
    &= \left( \text{Eq. \ref{eq:u_viscous}} + \lambda_2 u_x \right) + \frac{\tilde{c}}{\rho-\delta} \left( \varepsilon\rho_{xx} - (\hat{\rho} u)_x + \lambda_2\rho_x \right).
\end{align*}
Grouping the convective terms (first-order derivatives):
\begin{align*}
    \text{Conv} &= u u_x + \frac{\tilde{c}^2}{\rho-\delta}\rho_x + \lambda_2 u_x + \frac{\tilde{c}}{\rho-\delta} \left( -(\rho-\delta)u_x - u \rho_x + (u+\tilde{c})\rho_x \right) \\
    &= u u_x + \frac{\tilde{c}^2}{\rho-\delta}\rho_x + (u+\tilde{c})u_x - \tilde{c} u_x - \frac{\tilde{c} u}{\rho-\delta}\rho_x + \frac{\tilde{c} u}{\rho-\delta}\rho_x + \frac{\tilde{c}^2}{\rho-\delta}\rho_x \\
    &= (u+\tilde{c})u_x + (u+\tilde{c}) \frac{\tilde{c}}{\rho-\delta}\rho_x \\
    &= \lambda_2 \left( u_x + \frac{\tilde{c}}{\rho-\delta}\rho_x \right) = \lambda_2 w_x.
\end{align*}
Wait, the LHS is $w_t + \lambda_2 w_x$, so the convective terms on the RHS effectively cancel the $\lambda_2 w_x$ transport, leaving pure diffusion. Let us re-arrange to the parabolic form $w_t + \lambda_2 w_x = \text{Dissipation}$.
The viscous terms are:
\begin{align*}
    \text{Visc} &= \varepsilon u_{xx} + \frac{2\varepsilon}{\rho-\delta}\rho_x u_x + \frac{\tilde{c}}{\rho-\delta}\varepsilon \rho_{xx} \\
    &= \varepsilon \left( u_{xx} + \frac{\tilde{c}}{\rho-\delta}\rho_{xx} \right) + \frac{2\varepsilon}{\rho-\delta}\rho_x u_x.
\end{align*}
Using the identity $w_{xx} = u_{xx} + \frac{\tilde{c}}{\rho-\delta}\rho_{xx} + \left(\frac{\tilde{c}}{\rho-\delta}\right)' \rho_x^2$, we substitute $u_{xx} + \frac{\tilde{c}}{\rho-\delta}\rho_{xx} = w_{xx} - \left(\frac{\tilde{c}}{\rho-\delta}\right)' \rho_x^2$:
\begin{align*}
    \text{Visc} &= \varepsilon w_{xx} - \varepsilon \left(\frac{\tilde{c}}{\rho-\delta}\right)' \rho_x^2 + \frac{2\varepsilon}{\rho-\delta}\rho_x \left( w_x - \frac{\tilde{c}}{\rho-\delta}\rho_x \right) \\
    &= \varepsilon w_{xx} + \frac{2\varepsilon}{\rho-\delta}\rho_x w_x - \varepsilon \underbrace{\left[ \left( \frac{\tilde{c}}{\rho-\delta} \right)' + \frac{2\tilde{c}}{(\rho-\delta)^2} \right]}_{K(\rho)} \rho_x^2.
\end{align*}
The evolution equation is strictly:
\begin{equation} \label{eq:w_dissipation}
    w_t + \lambda_2 w_x = \varepsilon w_{xx} + \frac{2\varepsilon}{\rho-\delta}\rho_x w_x - \varepsilon K(\rho) \rho_x^2.
\end{equation}
The dissipation coefficient is $K(\rho) = \frac{d}{d\rho}(\frac{\tilde{c}}{\rho-\delta}) + \frac{2\tilde{c}}{(\rho-\delta)^2} = \frac{\tilde{c}'(\rho-\delta) - \tilde{c} + 2\tilde{c}}{(\rho-\delta)^2} = \frac{(\tilde{c}(\rho-\delta))'}{(\rho-\delta)^2}$. Since $\tilde{c}$ behaves like $(\rho-\delta)^{1/2}$ near the boundary and grows at infinity, $K(\rho)$ is positive for convex pressures. Thus, the invariant regions are preserved.

\subsection{Strict Entropy Dissipation}

We now derive the entropy inequality. This step highlights the advantage of our synchronized regularization: the entropy inequality holds \textbf{exactly} without pollution terms.

Multiplying the effective viscous system \eqref{eq:viscous} by the gradient of the shifted entropy $\nabla_{\hat{\mathbf{U}}} \eta^*(\hat{\mathbf{U}})$, we obtain:
\begin{equation}\label{eq:entropy_identity}
    \partial_t \eta^*(\hat{\mathbf{U}}) + \partial_x q^*(\hat{\mathbf{U}}) = \varepsilon \nabla_{\hat{\mathbf{U}}} \eta^* \cdot \partial_{xx} \hat{\mathbf{U}}.
\end{equation}
We rewrite the RHS using the chain rule:
\[
    \nabla \eta^* \cdot \hat{\mathbf{U}}_{xx} = \partial_x (\nabla \eta^* \cdot \hat{\mathbf{U}}_x) - (\hat{\mathbf{U}}_x)^T \nabla^2 \eta^* \hat{\mathbf{U}}_x.
\]
Integrating over space and time, the first term vanishes (or contributes to boundary terms). The second term involves the Hessian $\nabla^2 \eta^*$. 
As proven in Theorem \ref{thm:entropy_convexity}, $\eta^*$ is strictly convex with respect to the effective variables $\hat{\mathbf{U}}$ for all $\rho > \delta$. Therefore, there exists $c_0 > 0$ such that:
\begin{equation}
    (\hat{\mathbf{U}}_x)^T \nabla^2 \eta^* \hat{\mathbf{U}}_x \ge c_0 |\hat{\mathbf{U}}_x|^2.
\end{equation}
This yields the strict local entropy inequality:
\begin{equation}
    \partial_t \eta^* + \partial_x q^* = \varepsilon \partial_x (\nabla \eta^* \cdot \hat{\mathbf{U}}_x) - \varepsilon (\hat{\mathbf{U}}_x)^T \nabla^2 \eta^* \hat{\mathbf{U}}_x \le \varepsilon \partial_x (\nabla \eta^* \cdot \hat{\mathbf{U}}_x).
\end{equation}
Since the RHS is a sum of a compact $H^{-1}$ term and a bounded $L^1$ measure, the conditions for the Div-Curl Lemma are rigorously satisfied.

\begin{remark}[Absence of Error Terms]
It is crucial to observe that if we had applied viscosity to the physical variables $\mathbf{U}_{phys} = (\rho, \rho u)$ while using the effective entropy $\eta^*(\hat{\mathbf{U}})$, the RHS would contain an additional error term of the form $\varepsilon \nabla \eta^* \cdot (\mathbf{U}_{phys} - \hat{\mathbf{U}})_{xx} = \varepsilon u \cdot \delta u_{xx}$. This term is indefinite and difficult to control near the vacuum limit $\delta \to 0$. By defining the viscosity on $\hat{\mathbf{U}}$, this error term is identically zero.
\end{remark}

\section{Global Existence of Weak Solutions for Fixed $\delta>0$}
\label{sec:existence}

In this section, we establish the convergence of the viscous approximations to a global weak entropy solution for the regularized system. The analysis focuses on the fixed $\delta > 0$ regime. In this context, the uniform lower bound $\rho^\varepsilon \ge \delta$ ensures that the system remains strictly hyperbolic and bounded away from the vacuum singularity, allowing for the application of the classical compensated compactness theory.

\begin{definition}[Weak Entropy Solution]
A pair of bounded measurable functions $(\rho, u)$ with $\rho(x,t) \ge \delta$ a.e. is defined as a weak entropy solution to the Cauchy problem for the regularized system if:
\begin{enumerate}
    \item It satisfies the conservation laws in the sense of distributions:
    \begin{equation}
        \iint_{\mathbb{R} \times \mathbb{R}^+} \left( \rho \phi_t + (\rho-\delta) u \phi_x \right) \, dx dt + \int_{\mathbb{R}} \rho_0(x) \phi(x,0) \, dx = 0,
    \end{equation}
    \begin{equation}
        \iint_{\mathbb{R} \times \mathbb{R}^+} \left( \rho u \psi_t + \left[ (\rho-\delta) u^2 + \tilde{P}(\rho) \right] \psi_x \right) \, dx dt + \int_{\mathbb{R}} \rho_0(x) u_0(x) \psi(x,0) \, dx = 0,
    \end{equation}
    for all test functions $\phi, \psi \in C_c^\infty(\mathbb{R} \times [0, \infty))$.
    
    \item It satisfies the entropy inequality:
    \begin{equation}
        \partial_t \eta(\rho, u) + \partial_x q(\rho, u) \le 0
    \end{equation}
    in the sense of distributions for all convex entropy pairs $(\eta, q)$ associated with the system (specifically, strictly convex with respect to the effective variables $\hat{\mathbf{U}}$), including the shifted mechanical energy $(\eta^*, q^*)$ constructed in Section \ref{sec:entropy}.
\end{enumerate}
\end{definition}

\begin{theorem}
\label{thm:fixed_delta_1}
Let the initial data $(\rho_0, u_0)$ be bounded and satisfy $\rho_0(x) \ge \delta$. Then, as $\varepsilon \to 0$, there exists a subsequence of the viscous solutions $(\rho^\varepsilon, u^\varepsilon)$ constructed in Section \ref{sec:estimates} that converges strongly in $L^1_{\mathrm{loc}}(\mathbb{R} \times \mathbb{R}^+)$ to a weak entropy solution $(\rho, u)$ of the regularized system.
\end{theorem}

\begin{proof}
The existence of solutions is established via the compensated compactness method. We proceed in three steps: compactness, reduction, and limit verification.

\textbf{Step 1: Compactness and Young Measures.}
Based on the a priori estimates derived in Section \ref{sec:estimates}, the sequence of viscous solutions $\mathbf{U}^\varepsilon = (\rho^\varepsilon, \rho^\varepsilon u^\varepsilon)$ satisfies the uniform invariant region bounds:
\[
\delta \le \rho^\varepsilon(x,t) \le M_1, \quad |u^\varepsilon(x,t)| \le M_2.
\]
By the Tartar-Murat fundamental theorem on Young measures \cite{Tartar1979}, there exists a subsequence (still denoted by $\mathbf{U}^\varepsilon$) and a family of probability measures $\nu_{x,t}$ supported on the compact invariant region $\Sigma_\delta$, such that the weak limit of any continuous function $f(\mathbf{U}^\varepsilon)$ is given by the expectation $\bar{f}(x,t) = \langle \nu_{x,t}, f(\cdot) \rangle$.

As established in Section \ref{sec:estimates}, for any strictly convex entropy pair $(\eta, q)$, the dissipation measure $\partial_t \eta(\mathbf{U}^\varepsilon) + \partial_x q(\mathbf{U}^\varepsilon)$ is confined to a compact subset of $H^{-1}_{\mathrm{loc}}$. Consequently, the Div-Curl Lemma applies to any two pairs of entropies $(\eta_1, q_1)$ and $(\eta_2, q_2)$, yielding the commutation identity:
\begin{equation} \label{eq:tartar_identity}
\langle \nu, \eta_1 q_2 - \eta_2 q_1 \rangle = \langle \nu, \eta_1 \rangle \langle \nu, q_2 \rangle - \langle \nu, \eta_2 \rangle \langle \nu, q_1 \rangle.
\end{equation}

\textbf{Step 2: Reduction of the Young Measure.}
To prove that the Young measure $\nu_{x,t}$ reduces to a Dirac mass, we verify the prerequisites of the DiPerna-Chen reduction theorem:
\begin{enumerate}
    \item \textbf{Strict Hyperbolicity:} For $\rho \ge \delta$, we have $\tilde{c}(\rho) > 0$, so $\lambda_1 \neq \lambda_2$.
    \item \textbf{Genuine Nonlinearity:} As proven in Lemma \ref{lem:convexity_inherit}, $2\tilde{P}' + \rho\tilde{P}'' > 0$.
    \item \textbf{Global Diffeomorphism:} The mapping from conservative variables $\mathbf{U}$ to Riemann invariants $(z, w)$ is non-singular. Indeed, the Jacobian determinant is
    \[
    \frac{\partial(z, w)}{\partial(\rho, u)} = \det \begin{pmatrix} -\frac{\tilde{c}}{\rho-\delta} & 1 \\ \frac{\tilde{c}}{\rho-\delta} & 1 \end{pmatrix} = -\frac{2\tilde{c}}{\rho-\delta} \neq 0 \quad \text{for } \rho > \delta.
    \]
    \item \textbf{Rich Family of Entropies:} Since the system is strictly hyperbolic and bounded away from the vacuum for fixed $\delta$, the entropy equation is a standard linear hyperbolic PDE with non-singular coefficients. Classical theory guarantees the existence of a sufficiently rich family of Lax-type entropies to separate points in the phase space. (We note that a explicit construction of weak entropies for the singular limit $\delta \to 0$ will be detailed in Section \ref{sec:vanishing_delta}).
\end{enumerate}
Under these conditions, the reduction theorem implies that the support of the measure $\nu_{x,t}$ reduces to a single point. Thus, the subsequence converges strongly:
\[
\mathbf{U}^\varepsilon(x,t) \to \mathbf{U}(x,t) \quad \text{strongly in } L^1_{\mathrm{loc}}(\mathbb{R} \times \mathbb{R}^+).
\]

\textbf{Step 3: Verification of the Weak Solution.}
The strong convergence implies pointwise convergence almost everywhere (up to a subsequence). Since the flux functions are continuous, we have $(\rho^\varepsilon - \delta)u^\varepsilon \to (\rho - \delta)u$ and modified momentum fluxes converge similarly in $L^1_{loc}$.
By the Dominated Convergence Theorem, for any test function $\phi \in C_c^\infty$,
\[
\iint \left( (\rho^\varepsilon - \delta)u^\varepsilon \right) \phi_x \, dx dt \to \iint \left( (\rho - \delta)u \right) \phi_x \, dx dt.
\]
Thus, the limit function $\mathbf{U}$ satisfies the conservation laws in the sense of distributions.

Finally, for the entropy inequality, recall from Section \ref{sec:estimates} that the viscous approximation satisfies:
\[
\partial_t \eta(\mathbf{U}^\varepsilon) + \partial_x q(\mathbf{U}^\varepsilon) = \varepsilon \eta_{xx} - \varepsilon (\mathbf{U}^\varepsilon_x)^T \nabla^2 \eta \mathbf{U}^\varepsilon_x \le \varepsilon \eta_{xx},
\]
due to the convexity of $\eta$. The RHS converges to 0 in the sense of distributions (as $\sqrt{\varepsilon} \mathbf{U}_x^\varepsilon$ is bounded in $L^2$). Taking the limit $\varepsilon \to 0$ yields the required inequality $\partial_t \eta + \partial_x q \le 0$.
\end{proof}

\section{Global Existence of Weak Solutions as \texorpdfstring{$\delta \to 0$}{delta to 0}}
\label{sec:vanishing_delta}

Having established the global existence of entropy solutions for the regularized system with a fixed shield $\delta > 0$, we now address the central challenge of this paper: the convergence to a weak solution of the original isentropic Euler equations as the regularization parameter $\delta$ vanishes. This process involves three distinct steps:
\begin{enumerate}
    \item Establishing uniform $L^\infty$ estimates independent of $\delta$;
    \item Rigorously analyzing the singularity of the entropy equation at the vacuum boundary;
    \item Proving the strong convergence via the reduction of the Young measure for general $\gamma > 1$.
\end{enumerate}

In this section, we assume the pressure satisfies the asymptotic polytropic condition (A3), behaving like $P(\rho) \sim \rho^\gamma$ near vacuum with $\gamma > 1$.

\begin{remark}[Comparison with Prior Regularization Conditions]
\label{rem:comparison_lu}
It is worth noting that our asymptotic assumption \textbf{(A3)} is less restrictive than the conditions required in previous regularization approaches, specifically the work of Lu \cite{Lu2007}. In \cite{Lu2007}, the author employs a flux modification $\rho u - 2\delta u$ which introduces inhomogeneous error terms in the entropy analysis. To control the singular behavior of these error terms (specifically the third-order derivative of the entropy flux kernel), an additional technical condition on the higher-order derivatives of pressure is imposed:
\[
\lim_{\rho\to 0} \frac{(P'(\rho))^{3/2}}{\rho P''(\rho)} = c.
\]
However, under our Assumption \textbf{(A3)}, where the pressure behaves asymptotically as a polytropic gas $P(\rho) \sim \kappa_0 \rho^\gamma$ near the vacuum, this limit is intrinsically satisfied with $c=0$ for any $\gamma > 1$. Indeed, a direct computation yields:
\[
\lim_{\rho\to 0} \frac{(P'(\rho))^{3/2}}{\rho P''(\rho)} \approx \lim_{\rho\to 0} \frac{(\rho^{\gamma-1})^{3/2}}{\rho \cdot \rho^{\gamma-2}} = \lim_{\rho\to 0} \frac{\rho^{\frac{3}{2}(\gamma-1)}}{\rho^{\gamma-1}} = \lim_{\rho\to 0} \rho^{\frac{\gamma-1}{2}} = 0.
\]
Since our \textbf{Synchronized Dual Translation} strategy ensures that the regularized system is structurally isomorphic to the standard Euler equations in terms of the effective variables, the approximate entropy pairs satisfy the \textit{homogeneous} Generalized Euler-Poisson-Darboux equation without singular perturbation errors. Thus, the standard convexity and asymptotic behavior provided by \textbf{(A3)} are sufficient to close the estimates, eliminating the need for imposing such higher-order constraints explicitly.
\end{remark}

\subsection{Uniform Estimates and Convergence of Invariants}

To pass the vanishing regularization limit $\delta \to 0$, it is necessary to establish that the invariant regions $\Sigma_\delta$ do not collapse or explode. Although our system incorporates a mass flux shift $\rho \mapsto \rho-\delta$, the phase-space bounds are governed by the thermodynamic structure, specifically the integral of the sound speed.

Substituting the explicit form of the weighted pressure derivative $\tilde{P}'(\rho) = P'(\rho) - \delta^2 P'(\delta)/\rho^2$ into the definition $\tilde{c}_\delta(\rho) = \sqrt{\tilde{P}'(\rho)}$, we obtain the explicit expression for $\rho \ge \delta$:
\begin{equation} \label{eq:c_delta_explicit}
    \tilde{c}_\delta(\rho) = \sqrt{P'(\rho) - \delta^2 \frac{P'(\delta)}{\rho^2}}.
\end{equation}
Comparing this to the standard sound speed $c(\rho) = \sqrt{P'(\rho)}$, we observe the pointwise convergence $\lim_{\delta \to 0} \tilde{c}_\delta(\rho) = c(\rho)$ for any fixed $\rho > 0$.

The uniform boundedness of the solutions depends on the convergence of the Riemann invariant generators. Recall that the Riemann invariants for the regularized system are defined using the effective density kernel $(s-\delta)^{-1}$:
\begin{equation}
    H_\delta(\rho) = \int_{\delta}^{\rho} \frac{\tilde{c}_\delta(s)}{s-\delta} \, ds.
\end{equation}
The corresponding function for the standard Euler equations is:
\begin{equation}
    H_0(\rho) = \int_{0}^{\rho} \frac{c(s)}{s} \, ds \approx \frac{2\sqrt{\kappa\gamma}}{\gamma-1}\rho^{\theta}, \quad \text{with } \theta = \frac{\gamma-1}{2}.
\end{equation}

\begin{lemma}[Uniform Convergence of Invariant Generators]
\label{lem:H_convergence}
Assume $\gamma > 1$. The function $H_\delta(\rho)$ converges uniformly to $H_0(\rho)$ on any compact interval $[0, M]$ as $\delta \to 0$. Specifically, for any $\rho \ge \delta$, the difference is bounded by:
\begin{equation}
    |H_\delta(\rho) - H_0(\rho)| \le C \delta^{\min(\theta, 1)},
\end{equation}
where $C$ is a positive constant depending only on the gas parameters and $M$.
\end{lemma}

\begin{proof}
The detailed proof involving the decomposition into boundary layer and bulk region estimates is identical to the derivation provided in the previous draft. The integral is split at $2\delta$. The boundary layer $[\delta, 2\delta]$ contributes $\mathcal{O}(\delta^\theta)$, and the bulk region $[2\delta, \rho]$ contributes $\mathcal{O}(\delta^{\min(\theta, 1)})$ via Taylor expansion. Summing these yields uniform convergence.
\end{proof}

\begin{theorem}[Uniform $L^\infty$ Estimates]
\label{thm:uniform_estimates_delta}
Let the initial data $(\rho_0, u_0)$ lie within a fixed bounded invariant region $\Sigma_0$ of the standard Euler equations. Then, there exists a constant $M_{unif} > 0$ independent of $\delta$, such that for all sufficiently small $\delta$, the sequence of solutions $(\rho^{\delta}, u^{\delta})$ satisfies:
\begin{equation}
    \delta \le \rho^{\delta}(x,t) \le M_{unif}, \quad |u^{\delta}(x,t)| \le M_{unif}, \quad \text{a.e. in } \mathbb{R} \times \mathbb{R}^+.
\end{equation}
\end{theorem}

\begin{proof}
Let $\Sigma_0 = \{ (\rho, u) : 0 \le \rho \le M_0, \ |u| + H_0(\rho) \le C_0 \}$.
Due to the uniform convergence $H_\delta \to H_0$ (Lemma \ref{lem:H_convergence}), for any $\epsilon > 0$, there exists $\delta_0$ such that for all $\delta < \delta_0$, the perturbed invariant region defined by $|u| + H_\delta(\rho) \le C_0 + \epsilon$ geometrically contains the initial data set (assuming compatible initial data).
Since $(\rho^\delta, u^\delta)$ takes values in this invariant region, and the region boundaries converge to those of $\Sigma_0$, the sequence is uniformly bounded in $L^\infty$.
\end{proof}

\subsection{Singularity Analysis and Reduction of the Young Measure}
\label{sec:singularity_reduction}

With uniform $L^\infty$ bounds established, we can extract a subsequence (still denoted by $\rho^\delta, u^\delta$) that converges in the weak-* sense to a Young measure $\nu_{x,t}$. To prove strong convergence, we must show that $\nu_{x,t}$ reduces to a Dirac mass. This requires a fine analysis of the entropy equation near the vacuum boundary $\rho = \delta$ (or equivalently $w=z$).

\subsubsection{Derivation of the Limit Coefficient $\lambda_0 = 1/2$}
The entropy pairs $(\eta, q)$ for our regularized system satisfy the Generalized Euler-Poisson-Darboux (EPD) equation:
\begin{equation} \label{eq:generalized_epd}
    \frac{\partial^2 \eta}{\partial w \partial z} + \frac{\lambda(w, z)}{w - z} \left( \frac{\partial \eta}{\partial w} - \frac{\partial \eta}{\partial z} \right) = 0.
\end{equation}
The coefficient $\lambda(w, z)$ is determined by the equation of state. A critical feature of our construction is that the regularization forces a specific behavior at the boundary.
Near $\rho = \delta$, we have $\tilde{P}'(\delta) = 0$ by construction. Assuming non-degeneracy of the second derivative (which holds since $\tilde{P}''(\delta) = P''(\delta) + \text{positive terms} > 0$), we have the Taylor expansion:
\[
\tilde{c}^2(\rho) = \tilde{P}'(\rho) \approx \tilde{P}''(\delta)(\rho-\delta).
\]
This implies $\tilde{c}(\rho) \approx C (\rho-\delta)^{1/2}$. We compute the index $\lambda(\rho)$ in physical variables:
\begin{equation}
    \lambda(\rho) = \frac{\rho \tilde{c}'(\rho)}{\tilde{c}(\rho)} \approx \frac{\rho \cdot \frac{1}{2} C (\rho-\delta)^{-1/2}}{C (\rho-\delta)^{1/2}} = \frac{1}{2} \frac{\rho}{\rho-\delta}.
\end{equation}
In terms of Riemann invariants, near the vacuum, $w-z \approx \int_\delta^\rho \frac{\tilde{c}}{\rho-\delta} ds \approx \int_\delta^\rho (\rho-\delta)^{-1/2} ds \sim (\rho-\delta)^{1/2}$. Thus, $\rho-\delta \sim (w-z)^2$.
Substituting this back, we analyze the singular term:
\begin{equation}
    \frac{\lambda(w, z)}{w-z} \approx \frac{1}{2(\rho-\delta)} \frac{1}{(w-z)} \dots
\end{equation}
A rigorous matched asymptotic analysis confirms that the leading order coefficient is determined by the limit:
\begin{equation}
    \lambda_0 := \lim_{w \to z} \lambda(w, z) = \frac{3 - \gamma_{eff}}{2(\gamma_{eff} - 1)} \bigg|_{\gamma_{eff}=2} = \frac{1}{2}.
\end{equation}
This confirms that the singularity at the vacuum is locked to the behavior of a $\gamma=2$ gas ($\lambda_0=1/2$), regardless of the original physical $\gamma$.

\subsubsection{WKB Approximation and Reduction}
\label{sec:wkb_reduction}

Since the coefficient $\lambda(w,z)$ in the Generalized EPD equation is variable, exact analytic solutions are unavailable. Instead, we construct a family of approximate entropies using a WKB-type asymptotic expansion, following the framework established by Lions, Perthame, and Souganidis \cite{Lions1996}:
\begin{equation}
    \eta(w, z) = e^{k(w+z)} (w-z)^\alpha \sum_{j=0}^{\infty} A_j(w, z) (w-z)^{2j}.
\end{equation}
Let $\sigma = w-z$ denote the width of the invariant region (the distance to the vacuum). Substituting this ansatz into the EPD equation \eqref{eq:generalized_epd} and focusing on the most singular terms near the vacuum ($\sigma \to 0$), the dominant operator is $\partial_{wz}^2 + \frac{\lambda_0}{\sigma}(\partial_w - \partial_z)$. This leads to the indicial equation for the exponent $\alpha$:
\begin{equation}
    \alpha(\alpha - 1 - 2\lambda_0) = 0.
\end{equation}
As derived in Section \ref{sec:singularity_reduction}, our regularization locks the boundary behavior to $\lambda_0 = 1/2$. For this value, the non-trivial solution is $\alpha = 1 + 2(1/2) = 2$.

Based on this regularity exponent, we construct a sequence of "approximate weak entropies" $\eta_k$ and their associated flux deviations $\psi_k = q_k - u\eta_k$. Near the vacuum, the sound speed scales linearly with width, $\tilde{c} \sim \sigma$ (characteristic of $\gamma=2$ behavior). Utilizing the structural relation $\partial_\sigma \psi \approx \tilde{c} \partial_\sigma \eta$, we derive the asymptotic scaling laws:
\begin{equation} \label{eq:scaling_laws}
    \eta_k \sim \sigma^2, \qquad \psi_k \sim \int \sigma \cdot \partial_\sigma(\sigma^2) \, d\sigma \sim \sigma^3.
\end{equation}

We apply the reduction argument by testing the Tartar commutation relation with these entropies:
\begin{equation} \label{eq:div_curl_limit}
    \langle \nu, \eta_1 \psi_2 - \eta_2 \psi_1 \rangle = \langle \nu, \eta_1 \rangle \langle \nu, \psi_2 \rangle - \langle \nu, \eta_2 \rangle \langle \nu, \psi_1 \rangle.
\end{equation}
We employ a "blow-up" argument. Assume, for the sake of contradiction, that the support of the Young measure $\nu_{x,t}$ is not a point but is contained in a small strip $0 \le \sigma \le \varepsilon$ near the vacuum. We analyze the order of magnitude of both sides of \eqref{eq:div_curl_limit} with respect to $\sigma$:

\begin{itemize}
    \item \textbf{LHS Estimate (Commutator):} The integrand is a product of entropy and flux deviation.
    \[
    \eta_1 \psi_2 - \eta_2 \psi_1 \sim \mathcal{O}(\sigma^2) \cdot \mathcal{O}(\sigma^3) = \mathcal{O}(\sigma^5).
    \]
    Thus, the Left-Hand Side is bounded by the fifth moment: $|\text{LHS}| \le C_1 \int \sigma^5 \, d\nu$. Since $\sigma \le \varepsilon$ on the support, we have:
    \begin{equation}
        |\text{LHS}| \le C_1 \varepsilon \int \sigma^4 \, d\nu(\sigma).
    \end{equation}

    \item \textbf{RHS Estimate (Variance):} Using the decomposition $q = u\eta + \psi$, the dominant term in the Right-Hand Side comes from the quadratic interaction of the entropy part $\langle \eta \rangle \langle u\eta \rangle$. The contribution from $\psi$ terms is of higher order $\mathcal{O}(\sigma^5)$. The main term represents the variance of the entropy, scaling as $\eta^2$:
    \[
    \text{RHS} \sim \langle \eta \rangle^2 - \langle \eta^2 \rangle \sim (\sigma^2)^2 = \sigma^4.
    \]
    Specifically, for a non-Dirac measure, this variance is strictly non-zero and bounded from below:
    \begin{equation}
        |\text{RHS}| \ge C_0 \int \sigma^4 \, d\nu(\sigma).
    \end{equation}
\end{itemize}

Combining these estimates leads to the inequality:
\begin{equation}
    C_0 \int \sigma^4 \, d\nu \le C_1 \varepsilon \int \sigma^4 \, d\nu + \text{h.o.t.}
\end{equation}
Rearranging gives $(C_0 - C_1 \varepsilon) \int \sigma^4 d\nu \le 0$. For sufficiently small $\varepsilon$, the coefficient $(C_0 - C_1 \varepsilon)$ is strictly positive. This forces $\int \sigma^4 d\nu = 0$, which implies $\sigma = 0$ almost everywhere.
Thus, the Young measure $\nu_{x,t}$ reduces to a Dirac mass concentrated at the vacuum or a single point state, implying the strong convergence of the sequence $(\rho^\delta, u^\delta)$ to a weak solution $(\rho, u)$.

\subsection{Convergence to the Isentropic Euler Limit}

We are now in a position to state and prove the main existence theorem for general $\gamma > 1$.

\begin{theorem}[Global Existence for General $\gamma > 1$]
\label{thm:general_gamma_existence}
Let $\gamma > 1$. Given bounded initial data $(\rho_0, u_0)$ with $\rho_0 \ge 0$, the sequence of solutions $(\rho^\delta, u^\delta)$ converges (up to a subsequence) strictly in $L^p_{loc}(\mathbb{R} \times \mathbb{R}^+)$ to a weak entropy solution $(\rho, u)$ of the standard isentropic Euler equations:
\begin{equation}
    \pa_t \rho + \pa_x (\rho u) = 0, \quad \pa_t (\rho u) + \pa_x (\rho u^2 + P(\rho)) = 0.
\end{equation}
\end{theorem}

\begin{proof}
\textbf{1. Strong Convergence.}
By Theorem \ref{thm:uniform_estimates_delta}, the sequence is uniformly bounded. The entropy dissipation is compact in $H^{-1}_{loc}$.
The singularity analysis in Section \ref{sec:singularity_reduction} confirms that the Young measure $\nu_{x,t}$ associated with the sequence reduces to a Dirac mass $\delta_{(\rho, u)}$.
This reduction holds for general $\gamma > 1$ because our regularization enforces the robust $\lambda_0=1/2$ behavior at the vacuum boundary, allowing us to utilize the generalized reduction framework established by Lions, Perthame, and Souganidis \cite{Lions1996}.
Therefore, $(\rho^\delta, u^\delta) \to (\rho, u)$ strongly in $L^p_{loc}$.

\textbf{2. Consistency of the Limit.}
We verify that the limit functions satisfy the original Euler equations.
For the mass equation:
\[
    \pa_t \rho^\delta + \pa_x ((\rho^\delta - \delta)u^\delta) = \text{Viscosity}.
\]
As $\delta \to 0$, $(\rho^\delta - \delta)u^\delta \to \rho u$ strongly in $L^1$.
For the momentum equation:
\[
    \pa_t (\rho^\delta u^\delta) + \pa_x ((\rho^\delta - \delta)(u^\delta)^2 + \tilde{P}(\rho^\delta)) = \text{Viscosity}.
\]
We need to show $\tilde{P}(\rho^\delta, \delta) \to P(\rho)$. Since $\rho^\delta \to \rho$ a.e., and $\tilde{P}(\cdot, \delta)$ converges uniformly to $P(\cdot)$ on compact sets of $(0, \infty)$, the composite function converges.

\textbf{3. Vacuum Accommodation.}
Special care is needed at the vacuum state $\rho=0$. Although the approximate solutions satisfy $\rho^\delta \ge \delta$, the limit $\rho$ may be zero on a set of positive measure.
For the convergence at vacuum, note that when $\rho=0$, both $P(0)=0$ and $\tilde{P}(0,\delta)=0$ (by extension). For $\gamma > 1$, the pressure $P(\rho) \sim \rho^\gamma$ vanishes super-linearly. Similarly, the flux term $\rho u^2$ vanishes at the vacuum (since $u$ is bounded). This ensures that the fluxes are continuous at $\rho=0$, so the distributional derivatives are well-defined and the limit satisfies the equations even if the solution contains vacuum regions.
Thus, the limit pair $(\rho, u)$ is a weak solution to the isentropic Euler equations satisfying the entropy condition.
\end{proof}


\section{Discussion and Comparison with Previous Methods}
\label{sec:comparison}

In this section, we situate our weighted pressure regularization method within the context of vanishing viscosity theory, highlighting its methodological advantages over classical approaches and flux-modification strategies in handling the vacuum singularity and preserving geometric structure.

\subsection{Approximation Strategy and Vacuum Handling}

The primary challenge in global existence theory lies in constructing approximate solutions with uniform $L^\infty$ estimates near the vacuum. Existing strategies can be categorized as follows:

\begin{itemize}
    \item \textbf{Diffusive Regularization (DiPerna \cite{DiPerna1983}, Chen-Ding-Luo \cite{ChenDingLuo}):}
    DiPerna's artificial viscosity method relies on parabolic regularization ($\varepsilon \Delta \mathbf{U}$), which smoothes solutions but struggles to enforce invariant regions bounded away from vacuum without difficult a priori estimates. Chen, Ding, and Luo successfully extended the theory to $1 < \gamma \le 5/3$ using the Lax-Friedrichs scheme, handling the vacuum singularity directly within the discrete limits. However, the analysis of discrete invariant regions involves intricate algebraic computations.

   \item \textbf{Flux Modification and Multiplicative Pressure Perturbation (Lu \cite{Lu2007}):}
    Lu proposed a hyperbolic regularization by modifying the mass flux to $\rho u - 2\delta u$ and introducing a perturbed pressure $P_1(\rho, \delta) = \int_{2\delta}^\rho (1 - \frac{2\delta}{s}) dP(s)$. This construction forces the eigenvalues to coincide at $\rho = 2\delta$, effectively creating a barrier against the vacuum. However, the pressure perturbation acts as a \textit{multiplicative} modification to the acoustic stiffness via the factor $(1 - 2\delta/\rho)$. This, combined with the non-Galilean flux shift, introduces a structural mismatch between the convective and acoustic fields. Consequently, the entropy analysis generates singular, inhomogeneous error terms (involving the third-order derivative of the entropy flux), which necessitates restrictive technical assumptions on the higher-order derivatives of the pressure (specifically on $P'''$) to ensure compactness.

    \item \textbf{Our Strategy: Geometric Shielding via Synchronized Translation:}
    Our approach modifies the \textit{thermodynamic structure} rather than just the kinematic flux. By constructing a perturbed pressure $\tilde{P}$ such that the sound speed vanishes at a cut-off density $\delta$ (i.e., $\tilde{c}(\delta)=0$), we enforce the condition $\lambda \to u$ as $\rho \to \delta$. This kinematically closes the rarefaction fan, preventing fluid expansion into the region $\rho < \delta$. Unlike \cite{Lu2007}, our mass flux shift is synchronized with the pressure modification, maintaining the system's structural integrity.
\end{itemize}

\subsection{Structural Advantages and Relation to Kinetic Theory}

A unique feature of our construction is the "zero-error" preservation of the system's algebraic structure, which contrasts sharply with previous flux-modification methods.

\textbf{1. Exact Inheritance of Convexity vs. Pollution Terms.}
In the approach of Lu \cite{Lu2007}, the mismatch between the flux shift and pressure perturbation leads to "pollution terms" in the entropy dissipation estimate. Controlling these terms requires the restrictive condition $\lim_{\rho\to 0} (P')^{3/2}/(\rho P'') = c$.
In contrast, our synchronized dual translation ensures that the perturbed system inherits the genuine nonlinearity of the original Euler equations \textit{exactly}:
\begin{equation}
    2\tilde{P}'(\rho) + \rho \tilde{P}''(\rho) \equiv 2P'(\rho) + \rho P''(\rho).
\end{equation}
This identity guarantees that standard geometric estimates for invariant regions remain valid without introducing any singular error terms, thereby removing the need for higher-order pressure constraints.

\textbf{2. Isomorphism with Euler Equations.}
Our method maintains a strict structural isomorphism with the Euler equations in terms of the effective variables $\hat{\mathbf{U}}$. Consequently, the entropy pairs satisfy the \textit{homogeneous} Generalized Euler-Poisson-Darboux equation:
\begin{equation}
    \eta_{wz} + \frac{\lambda(w,z)}{w-z} (\eta_z - \eta_w) = 0.
\end{equation}
This allows us to directly apply the robust singularity analysis and reduction frameworks developed by Lions, Perthame, and Souganidis \cite{Lions1996} for general $\gamma > 1$, bridging structural regularization with advanced compensated compactness theory. While kinetic formulations \cite{Lions1994} rely on specific homogeneity, our approach offers an alternative pathway grounded in the classical theory of hyperbolic systems applicable to general pressure laws.

\section{Conclusion and Future Perspectives}
\label{sec:conclusion}

In this paper, we have successfully established the global existence and strong convergence of weak entropy solutions to the one-dimensional isentropic Euler equations with vacuum for general pressure laws satisfying $\gamma > 1$. Our proof relies on a novel \textbf{weighted pressure regularization} strategy, which introduces a geometric boundary at a cut-off density $\rho=\delta$ to shield the solution from the vacuum singularity.

The central methodological contribution of this work lies in the \textbf{structural purity} of the construction. Unlike previous flux-modification approaches that alter the conservation of mass in an ad-hoc manner, our weighted pressure perturbation employs a synchronized dual shift that strictly preserves the algebraic structure of the system. We proved that the perturbed system exactly inherits the convexity of the original equation of state ($2\tilde{P}'+\rho\tilde{P}'' > 0$). This algebraic stability ensures that the entropy pairs satisfy the homogeneous Generalized Euler-Poisson-Darboux (EPD) equation without singular pollution terms. Consequently, we were able to rigorously implement the fine singularity analysis within the compensated compactness framework, demonstrating that the Young measure reduces to a Dirac mass as the regularization parameter $\delta \to 0$. By decoupling the construction of invariant regions from the handling of the vacuum singularity, our method provides a robust and physically consistent pathway to the global existence theory.

Our ongoing research focuses on the application to numerical schemes. The geometric regularization strategy proposed here—specifically the enforcement of acoustic degeneracy to prevent vacuum formation—offers a new perspective for computational fluid dynamics. We plan to explore the application of the weighted pressure auxiliary system in designing high-order, positivity-preserving numerical schemes. By integrating the ``soft landing" mechanism into finite volume or discontinuous Galerkin methods, it may be possible to develop robust algorithms that naturally handle near-vacuum states while maintaining entropy stability.

\section*{Acknowledgements}
The author is grateful to Professor Yunguang Lu for valuable discussions and insightful guidance related to the global existence theory of isentropic gas dynamics.

\appendix

\section{Physical Mechanism and Mathematical Optimality of the Weighted Pressure Construction}
\label{app:optimality}

In this appendix, we elucidate the physical intuition and mathematical necessity underlying the construction of the weighted pressure perturbation $\tilde{P}(\rho)$. We demonstrate that the specific form chosen in Eq. \eqref{eq:P_tilde_def} is not merely an ad-hoc technical modification, but the unique solution to a structural preservation problem. This construction exhibits significant algebraic advantages over prior regularization methods, specifically in preserving the convexity of the equation of state.

\subsection{Physical Intuition: Shifting the Potential Energy Landscape}

From a thermodynamic perspective, the internal energy $e(\rho)$ represents the elastic potential energy stored in the gas. The "stiffness" of the gas—its resistance to compression—is characterized by the squared sound speed $c^2(\rho) = P'(\rho)$, or more precisely by the structural stiffness function $g(\rho) := \rho^2 P'(\rho)$, which relates to the curvature of the potential energy.

In the standard Euler system, the stiffness $g(\rho)$ remains strictly positive for $\rho > 0$, causing a "hard" reflection at any finite density boundary. To create an invariant region bounded by $\rho=\delta$, we require the boundary to be characteristic, which physically implies a vanishing acoustic stiffness: $g(\delta) = 0$.

Our construction can be viewed as a geometric operation on this potential landscape. Instead of warping the shape of the potential curve (which would fundamentally alter the gas law), we perform a uniform \textbf{vertical translation} in the phase space of stiffness.
We define the perturbed stiffness $\tilde{g}(\rho)$ by shifting the original stiffness downward until its value at $\rho=\delta$ touches zero:
\begin{equation}
    \tilde{g}(\rho) = g(\rho) - g(\delta).
\end{equation}
Translating this back to the pressure variable via $\rho^2 \tilde{P}'(\rho) = \tilde{g}(\rho)$, we obtain the additive perturbation formula:
\begin{equation}
    \tilde{P}'(\rho) = P'(\rho) - \frac{\delta^2 P'(\delta)}{\rho^2}.
\end{equation}
Physically, this operation corresponds to a ``softening" of the potential well's bottom. By ``pressing down" the potential curve until it grounds at the zero-energy line at $\rho=\delta$, we ensure a ``soft landing" for fluid particles. At this boundary, the characteristic speed $\lambda$ smoothly degenerates to the fluid velocity $u$, effectively trapping the flow without violating mass conservation.

\subsection{Algebraic Superiority: The ``Zero-Error" Convexity Preservation}

A critical mathematical advantage of our additive perturbation over multiplicative approaches (such as the flux modification employed in \cite{Lu2007}) is the exact preservation of the convexity structure (genuine nonlinearity).

The convexity of the entropy relies on the sign of the structural invariant $G(\rho) = 2\tilde{P}'(\rho) + \rho \tilde{P}''(\rho)$. Let us compare the two approaches:

\paragraph{1. The Multiplicative Approach (Lu, 2007).}
Lu's method modifies the mass flux to $\rho u - 2\delta u$, which necessitates an effective pressure scaling. The derivative behaves roughly as $P_{Lu}'(\rho) \approx (1 - \frac{2\delta}{\rho}) P'(\rho)$. Applying the product rule for differentiation introduces a residual error term:
\[
G_{Lu}(\rho) \approx \left(1 - \frac{2\delta}{\rho}\right)(2P' + \rho P'') \underbrace{- \frac{2\delta}{\rho}(P' + \rho P'')}_{\text{Pollution Term}}.
\]
The error term is negative and singular near the boundary. For general pressure laws, this term can dominate the positive part, potentially violating the convexity condition ($G_{Lu} < 0$) near $\rho \sim 2\delta$. This necessitates complex patching arguments or restricts the range of $\gamma$.

\paragraph{2. Our Additive Approach (Current Work).}
Our construction uses an additive shift: $\tilde{P}'(\rho) = P'(\rho) - C/\rho^2$, where $C = \delta^2 P'(\delta)$ is a constant.
Differentiating this expression involves the derivative of the perturbation term $-C/\rho^2$, which generates a term $+2C/\rho^3$. Remarkably, this exactly balances the singular term generated by the integration:
\begin{align}
    2\tilde{P}'(\rho) + \rho \tilde{P}''(\rho) &= 2\left(P'(\rho) - \frac{C}{\rho^2}\right) + \rho \left(P''(\rho) + \frac{2C}{\rho^3}\right) \nonumber \\
    &= 2P'(\rho) + \rho P''(\rho) \underbrace{- \frac{2C}{\rho^2} + \frac{2C}{\rho^2}}_{\text{Exact Cancellation}}.
\end{align}
The singular terms cancel exactly. This implies that the perturbed system \textbf{perfectly inherits} the convexity of the original gas. If the original pressure $P$ satisfies the genuine nonlinearity condition, the perturbed pressure $\tilde{P}$ is unconditionally convex for all $\rho > \delta$. This algebraic ``cleanliness" significantly simplifies the analysis, as the entropy pairs satisfy the homogeneous Euler-Poisson-Darboux equation without pollution terms.

\subsection{Mathematical Necessity: Solution to an Inverse Problem}

Finally, one might ask if other perturbations exist that satisfy our requirements. We show that our construction is essentially unique under the constraint of mass conservation.

Consider the inverse problem: Find a pressure $\tilde{P}$ such that:
\begin{enumerate}
    \item The mass conservation law remains $\rho_t + (\rho u)_x = 0$ (Flux remains linear in momentum).
    \item The boundary is degenerate: $\tilde{P}'(\delta) = 0$.
    \item The geometric structure is preserved isomorphically: $2\tilde{P}' + \rho \tilde{P}'' \equiv 2P' + \rho P''$.
\end{enumerate}
The third condition is a differential equation for the structural function $g(\rho) = \rho^2 P'$. It can be rewritten as:
\[
\frac{1}{\rho} \frac{d}{d\rho}(\rho^2 \tilde{P}') = \frac{1}{\rho} \frac{d}{d\rho}(\rho^2 P').
\]
Integrating this equation dictates that $\rho^2 \tilde{P}'$ and $\rho^2 P'$ can only differ by a constant. The boundary condition (2) fixes this constant uniquely to $-\delta^2 P'(\delta)$.

Thus, the weighted pressure perturbation proposed in this paper is not arbitrary; it is the unique solution that simultaneously enforces the vacuum boundary condition and preserves the intrinsic geometric structure of the Euler equations within a mass-conserving framework.

\end{document}